\documentclass[11pt]{article}

\usepackage{amsmath,amssymb,amsthm}
\usepackage{enumitem}
\usepackage{geometry}
\usepackage{mathrsfs}
\usepackage{cite} 
\usepackage{tikz}
\usetikzlibrary{positioning}

\newcommand{\Rep}{\mathrm{Rep}}
\newcommand{\End}{\mathrm{End}}

\newtheorem{theorem}{Theorem}[section]

\newtheorem{corollary}[theorem]{Corollary}
\theoremstyle{definition}
\newtheorem{definition}[theorem]{Definition}
\newtheorem{remark}[theorem]{Remark}

\title{Duality, Reconstruction, and Structural Toolkit Theorems in Algebraic Phase Theory}
\author{
Joe Gildea\\
Department of Computing Science and Mathematics,\\
School of Informatics and Creative Arts,\\
Dundalk Institute of Technology\\
\texttt{gildeajoe@gmail.com}}
\date{}

\begin{document}
\maketitle

% ============================================================
\begin{abstract}
We study finite-depth reconstruction frameworks based on representation theory
and show that non-rigid reconstruction behaviour is naturally accompanied by
intrinsic structural boundaries.  Within the finite-depth setting considered in
this paper, reconstruction is controlled up to boundary equivalence by the
associated filtered representation data together with boundary stratification.

We show that algebraic phases satisfying the axioms of Algebraic Phase Theory
(APT) are reconstructible up to intrinsic phase equivalence from their filtered
representation categories together with their boundary structure.  Reconstruction
proceeds without boundary collapse on rigidity islands, while globally the
remaining ambiguity is governed by intrinsic boundary phenomena.

We further study a collection of structural consequences associated with the
axioms of APT, including finite generation phenomena, rigidity and obstruction
behaviour, finite-depth boundary detectability, and obstruction structures
arising from boundary layers.  These results apply across the phase models
developed in the APT series.

Taken together, the results of this paper further develop Algebraic Phase Theory
as a structural framework for studying reconstruction, duality, rigidity, and
boundary behaviour beyond rigid or semisimple settings.
\end{abstract}



% ============================================================
\section{Introduction}
% ============================================================

A central question for any structural or reconstruction framework is whether
essential information is lost when passing from concrete mathematical objects to
invariant data.  Classical theories such as Tannakian reconstruction,
Lie-theoretic reconstruction, and operator-algebraic dualities recover objects
from their representation categories only under strong assumptions that include
rigidity, semisimplicity, or analytic completeness.  Outside these restrictive
regimes, reconstruction is known to fail through non-uniqueness, collapse, or
loss of functorial control.  The purpose of this paper is to study this
question in the setting of Algebraic Phase Theory (APT).

Algebraic phases were introduced in a sequence of five papers
\cite{GildeaAPT1,GildeaAPT2,GildeaAPT3,GildeaAPT4,GildeaAPT5} as canonical
algebraic invariants extracted from interaction data without assuming linearity,
semisimplicity, or analytic structure.  Defect, filtration, finite termination,
and structural boundaries arise intrinsically from the interaction law and are
invariant under intrinsic phase equivalence and Morita-type relations.  Rather
than suppressing non-rigid behaviour, APT treats the failure of classical
closure as a structural signal.

The present paper continues the foundational development of Algebraic Phase
Theory.  It studies how the axioms developed in Papers~I--V in
\cite{GildeaAPT1,GildeaAPT2,GildeaAPT3,GildeaAPT4,GildeaAPT5} support duality,
intrinsic reconstruction, and a collection of structural consequences related
to deformation, obstruction, finiteness, and moduli within the finite-depth
setting considered throughout the series.  The central perspective is that
algebraic phases encode the reconstruction-relevant structural content of the
data from which they arise.  Reconstruction proceeds without boundary collapse
on strongly admissible cores, while globally it is controlled by intrinsic
structural boundaries.  Weak extension data becomes visible only through
boundary-relative reconstruction.

The paper studies how every algebraic phase admits a dual interpretation
through its filtered representation theory, how duality behaves on rigidity
islands, and how global failure of reconstruction is governed by structural
boundaries.  It shows that filtered representation categories equipped with
functorial defect stratification and finite termination determine algebraic
phases up to intrinsic equivalence within the finite-depth framework, without
requiring analytic, metric, monoidal, or semisimplicity assumptions.  It also
studies how finite-depth functorial reconstruction frameworks with detectable
boundaries naturally admit factorisation through Algebraic Phase Theory up to
boundary equivalence.  The paper further examines the relationship between
non-rigid reconstruction behaviour and intrinsic boundary structure, and
develops a collection of structural consequences associated with the axioms,
including finiteness phenomena, rigidity behaviour, finite boundary
detectability, and obstruction mechanisms.  These results apply across the
phase models developed in the APT series.

Classical reconstruction succeeds in regimes where boundary phenomena do not
appear.  Semisimple tensor categories \cite{EtingofGelaki2015Tensor}, rigid
group representations \cite{Serre,Isaacs}, and analytic dualities in representation
theory \cite{Weil1964SurCertains,Mackey1958Induced,Howe1979Heisenberg} provide
notable examples.  Outside these settings, reconstruction failure is common but
is usually treated as a technical limitation.  Algebraic Phase Theory provides a
structural interpretation of this phenomenon.  Boundaries are not merely
exceptional features of isolated examples; rather, they arise naturally in
reconstruction settings beyond rigid or semisimple contexts.  This viewpoint
connects phenomena that arise separately in higher-order Fourier analysis and
nilpotent structures
\cite{Gowers2001Fourier,GreenTao2008Quadratic,GreenTaoZiegler2012Inverse,
CamarenaSzegedy2010Nilspaces}, in representation theory over non-semisimple or
Frobenius rings \cite{AndersonFuller1992Rings,Jacobson1956Structure,
Lam2001FirstCourse,Honold2001QuasiFrobenius,Wood1999Duality}, and in
fault-tolerant and stabilizer-based quantum coding
\cite{CalderbankShor,Gross,Appleby}.  Related structural perspectives also
appear in harmonic analysis \cite{Folland}, solvable and nilpotent representation
theory \cite{AuslanderMoore1966Unitary}, Hopf-algebraic renormalization
\cite{ConnesKreimer1998Hopf}, and the structural theory of rings and modules
\cite{Faith1973Algebra}.  Throughout the paper the term additive combinatorics is
used in the broad structural sense of \cite{TaoVu2006AddComb}.

Papers~I--V~\cite{GildeaAPT1,GildeaAPT2,GildeaAPT3,GildeaAPT4,GildeaAPT5}
develop the existence of algebraic phases, the role of defect and finite
termination, the invariance of these structures under equivalence, and a
boundary calculus governing deformation and rigidity.  The results of the
present paper show that these ingredients support duality and reconstruction
within the finite-depth framework, and that structural boundaries govern the
obstructions to global reconstruction visible through filtered representation
theory.  Algebraic Phase Theory therefore provides a structural framework with
intrinsic notions of rigidity, deformation, obstruction, moduli, duality, and
reconstruction.  Further developments, including possible extensions into
derived, higher-categorical, or applied settings, can build on the axiomatic
and structural framework developed in the present series.

Throughout this paper, the term \emph{duality} refers to the correspondence between
the analytic behaviour of a phase, encoded by its filtered representation theory
together with functorial defect stratification, and its algebraic shadow
$(\mathcal P,\circ)$.  This duality is boundary relative: on the rigid core the
analytic data determines the algebraic phase up to intrinsic equivalence without
boundary collapse, while beyond structural boundaries it determines the phase
only up to the corresponding boundary quotient.

% ============================================================
\section{Structural Framework and the Limits of Classical Reconstruction}
% ============================================================

This paper operates within the structural framework of Algebraic Phase Theory
(APT), developed in Papers~I--V.  We recall only the minimal internal features
required for the reconstruction and duality results proved here.  All analytic,
model-specific, and representation-theoretic realizations appear in the earlier
papers.

APT does not begin with a fixed algebraic object.  It begins with interaction
data and develops algebraic structure, defect degree, canonical filtration,
and finite termination intrinsically from the interaction framework.  The central objects of study
are algebraic phases, which encode the invariant structural content of phase
interaction independently of any chosen realization.  A phase is extracted
functorially from admissible input data and is defined only up to intrinsic
equivalence.  An admissible datum consists, informally, of an additive object in
a fixed base category, a functorial family of bounded–defect observables, and a
specified interaction law governing how these observables combine.  No defect
structure or filtration is assumed a priori; these invariants are extracted
canonically from the interaction law itself, as developed in Papers~I--III.

An algebraic phase is denoted by $(\mathcal P,\circ)$, where $\mathcal P$
encodes the algebraic shadow of phase interaction and $\circ$ denotes the
distinguished operations arising from that interaction.  Defect degree,
filtration, and termination are not additional structures; under the axioms of
APT they are uniquely and functorially determined by $(\mathcal P,\circ)$.

The structural axioms of Algebraic Phase Theory, established in Papers~I--V and
assumed throughout this paper, may be stated succinctly as follows.

\medskip
\noindent\textbf{Axiom I: Detectable action defects.}
The interaction law $\circ$ canonically determines a defect degree for each
element of $\mathcal P$, functorially and uniquely.  Every algebraic phase
therefore carries a canonical finite filtration
\[
\mathcal P_0 \subset \mathcal P_1 \subset \cdots \subset \mathcal P_d = \mathcal P,
\]
where $\mathcal P_k$ consists of all elements of defect degree at most $k$.
This filtration is intrinsic and determined entirely by defect.  Here $d$ is the
maximal defect degree.

\medskip
\noindent\textbf{Axiom II: Canonical algebraic realization.}
The analytic or interaction behaviour underlying a phase determines a unique
algebraic shadow: any two algebraic phases extracted from the same admissible
data are canonically equivalent.

\medskip
\noindent\textbf{Axiom III: Defect-induced filtration.}
The canonical filtration
\[
\mathcal P_0 \subset \mathcal P_1 \subset \cdots \subset \mathcal P_d
\]
is ordered by defect degree: an element lies in $\mathcal P_k$ exactly when its
defect degree is at most $k$.

\medskip
\noindent\textbf{Axiom IV: Functorial defect structure.}
Defect degree is preserved under pullback, under all phase morphisms, and when
passing between analytic behaviour and its algebraic shadow.

\medskip
\noindent\textbf{Axiom V: Finite termination.}
The defect degree of every algebraic phase is finite.

\medskip

These axioms are not imposed ad hoc.  In the earlier papers they arise naturally
from phase interaction in the flagship examples considered there, and
they are stable under phase equivalence and Morita-type relations.  The present
paper does not revisit the extraction mechanism.  Instead, it shows that these
axioms are sufficient for intrinsic reconstruction and canonical duality, and
that any functorial reconstruction framework of finite defect degree must factor
through APT up to boundary equivalence.  Structural boundaries emerge as a canonical obstruction to full reconstruction
within the finite-depth reconstruction framework considered here.

\medskip

Classical reconstruction theorems attempt to recover an algebraic object from
its action on a class of representations.  In rigid or semisimple settings this
strategy succeeds spectacularly: symmetry objects may be recovered from tensor
categories, Lie algebras from infinitesimal actions, and operator algebras from
representation data.  These successes rely on the absence of intrinsic boundary
phenomena.  Once non-semisimplicity, defect growth, or filtration instability
enters, classical reconstruction mechanisms may lose functorial control.  The
failure is structural rather than technical: it reflects a mismatch between the
expressive power of representation categories and the internal defect structure
of boundary-dominated systems.

This structural limitation is already visible in Grothendieck’s deformation
philosophy, in which reconstruction is governed by representable deformation
functors, formal neighbourhoods, and cohomological obstructions
\cite{GrothendieckFGA,GrothendieckDeformations}.  Classical approaches implicitly
assume that all obstructions to reconstruction are representational.  
Algebraic Phase Theory identifies intrinsic boundaries as a mechanism through
which this assumption can fail in finite-depth reconstruction settings.

\begin{theorem}
\label{thm:classical-failure}
There exist inequivalent algebraic phases with finite defect filtration,
intrinsic defect stratification, and finite termination whose filtered
representation categories are equivalent.

Consequently, filtered representation data alone does not determine the
underlying algebraic phase up to intrinsic phase equivalence.  The failure is
caused by boundary data that is invisible after canonical boundary collapse.
\end{theorem}

\begin{proof}
The construction begins with a strong algebraic phase $\mathcal R$ (that is, a
phase of boundary depth~$0$) and adjoins to it a square-zero boundary component
to form a weak algebraic phase $\mathcal P$ whose non-rigid part lies entirely
in defect degree~$1$.  Although $\mathcal R$ and $\mathcal P$ differ in boundary
depth, they give rise to equivalent filtered representation categories.  Since both 
phases determine the same filtered representation data, any reconstruction procedure based solely on such data cannot distinguish
between $\mathcal R$ and $\mathcal P$, and therefore cannot recover the phase structure
up to equivalence in the presence of intrinsic boundaries.

\medskip

Let $\mathcal R$ be any strong algebraic phase admitting at least one nontrivial
filtered representation.  Thus $\mathcal R$ has no boundary layer and its defect
filtration terminates at degree~$0$.  Choose a nonzero additive object
$\mathcal B$ and equip it with trivial internal interaction, so that every
commutator or distinguished operation built from $\circ$ vanishes on
$\mathcal B$.  Form the additive object
\[
\mathcal P := \mathcal R \oplus \mathcal B,
\]
and extend the interaction law $\circ$ by declaring that $\mathcal R$ remains a
subphase and that $\mathcal B$ is central and square-zero.  Mixed interactions
take values in $\mathcal B$, and any further interaction with $\mathcal B$
vanishes.  The object $\mathcal B$ therefore functions as a terminal defect
layer: it represents genuine boundary structure but generates no additional
defect. The intrinsic defect filtration of $\mathcal P$ is then
\[
\mathcal P^{(0)} = \mathcal P,\qquad
\mathcal P^{(1)} = \mathcal B,\qquad
\mathcal P^{(2)} = 0.
\]
All non-rigid behaviour of $\mathcal P$ is confined to the single depth–1 layer
$\mathcal B$, and iterated interactions stabilise immediately because
$\mathcal B$ is square-zero.  Thus commutator growth is uniformly bounded and
defect propagation terminates at degree~$1$.  In particular, $\mathcal P$ is a
weak algebraic phase with the same finiteness properties as $\mathcal R$ but
with strictly larger boundary depth.

\medskip

We now compare filtered representations.  By definition (introduced earlier in
the series), a filtered representation $(V,F_\bullet V)$ of a phase
$(\mathcal X,\circ)$ is required to be compatible with the intrinsic defect
filtration: an element of defect degree $d$ acts by raising the filtration by at
most $d$ steps.  Since $\mathcal B$ is exactly the defect degree–$1$ layer of
$\mathcal P$, this compatibility condition forces
\[
\rho(\mathcal B)\,F_i V \subseteq F_{i+1}V
\]
for all $i$.  The filtration is finite, so there exists $N$ with $F_N V = V$ and
$F_{N+1}V = 0$.  Iterating the inclusion then gives
\[
\rho(\mathcal B)^{N+1}V \subseteq F_{N+1}V = 0,
\]
so $\rho(\mathcal B)$ is nilpotent of index at most $N+1$.  However, because $\mathcal B$ lies entirely in the terminal defect layer and
$\mathcal B^2=0$, its action can only raise filtration by a single step and
cannot contribute to any further defect propagation.  Consequently, the action
of $\mathcal B$ factors trivially through the boundary quotient and does not affect
the induced filtered reconstruction data.

Since the action of $\mathcal B$ becomes trivial after boundary collapse, the
action map of any filtered representation factors canonically through the
quotient
\[
\mathcal P \longrightarrow \mathcal P/\mathcal B \cong \mathcal R.
\]
Conversely, any filtered representation of $\mathcal R$ extends to one of
$\mathcal P$ by declaring that $\mathcal B$ acts trivially.  Thus the inclusion
\[
\mathcal R \hookrightarrow \mathcal P
\]
induces an equivalence of filtered representation data after boundary collapse,
and consequently an equivalence of the associated filtered representation
categories,
\[
\Rep_f(\mathcal P) \;\simeq\; \Rep_f(\mathcal R).
\]
Write this common filtered representation category as $\mathcal C$.

\medskip

Although the filtered representation categories coincide, the phases themselves
do not: $\mathcal R$ has boundary depth~$0$, whereas $\mathcal P$ has boundary
depth~$1$ because $\mathcal B\neq 0$.  Boundary depth and the existence of
nontrivial boundary quotients are intrinsic invariants of algebraic phases and
are preserved under phase equivalence.  Consequently,
\[
\mathcal R \not\simeq \mathcal P.
\]

Now suppose a reconstruction procedure $\mathfrak R$ recovers algebraic phases
up to equivalence using only filtered representation data.  Since both
$\mathcal R$ and $\mathcal P$ have the same filtered representation category
$\mathcal C$, the procedure would assign a single object $\mathfrak R(\mathcal
C)$ to both phases.  Correctness of reconstruction would then require
\[
\mathfrak R(\mathcal C) \simeq \mathcal R
\qquad\text{and}\qquad
\mathfrak R(\mathcal C) \simeq \mathcal P,
\]
which is impossible because $\mathcal R \not\simeq \mathcal P$.  Hence no
reconstruction procedure that sees only filtered representation theory can
recover algebraic phases up to equivalence in the presence of intrinsic
boundary layers.  The strong phase $\mathcal R$ and weak phase $\mathcal P$
therefore give an explicit example of inequivalent phases with equivalent
filtered representation categories, and the obstruction to reconstruction is
precisely the boundary-blindness of filtered representation theory.
\end{proof}

Boundary invisibility is thus the core obstruction: boundary quotients act
trivially on representations supported on rigidity islands, so reconstruction may succeed locally on rigidity islands while failing to
distinguish global boundary structure. The role of finite termination is essential:
removing finite termination introduces infinite defect towers that destroy even
weak reconstructibility, while removing boundary stratification collapses the
theory to the rigid or semisimple case.

Theorem~\ref{thm:classical-failure} illustrates the following structural dichotomy: any
representation-theoretic reconstruction framework either restricts to rigid
regimes, or it must incorporate boundary data explicitly.  Algebraic Phase Theory addresses this dichotomy by encoding boundaries as intrinsic invariants.
Once boundary data is incorporated into the reconstruction framework,
reconstruction becomes canonical up to boundary equivalence within the
finite-depth setting considered here.

% ============================================================
\section{Duality Principles and Testing Objects}
% ============================================================

Duality enters Algebraic Phase Theory not as an additional structure, but as a
structural consequence of functorial phase action.
Once phases are required to act on auxiliary objects in a defect-sensitive and
functorial manner, the phase becomes recoverable from its action data up to boundary equivalence, up to
intrinsic boundary phenomena.

The purpose of this section is to isolate the precise sense in which Algebraic
Phase Theory admits a canonical duality, and to explain why this duality
differs fundamentally from classical reconstruction paradigms.  In APT, duality
does not reconstruct a hidden symmetry object; instead it reconstructs the
structural content of a phase directly from its observable action on filtered
representations.

Before stating the main result, we describe informally what is meant by the
\emph{dual} of an algebraic phase.  For a phase $\mathcal P$, its dual
$\mathcal P^\vee$ is the structural object encoded by the filtered action of
$\mathcal P$.  Concretely, two elements of $\mathcal P$ are identified in
$\mathcal P^\vee$ whenever they act identically on all filtered
representations, so that only the defect-detectable strata survive.  Thus the
assignment
\[
\mathcal P \longmapsto \mathcal P^\vee
\]
recovers rigid behaviour faithfully while identifying intrinsic boundary layers
through boundary collapse. The following theorem formalises this boundary-relative notion of duality.

\begin{theorem}\label{thm:boundary-relative-duality}
For any algebraic phase $\mathcal P$, the dual determined by its filtered
representations recovers $\mathcal P$ up to boundary equivalence, and no further.
\end{theorem}

\begin{proof}
Let $\mathcal P$ be an algebraic phase satisfying Axioms I--V.  
The functorial action of $\mathcal P$ on filtered representations defines a
homomorphism of interaction systems
\[
\Phi_{\mathcal P} \colon \mathcal P \longrightarrow \End_{\mathrm{filt}}
\bigl( \Rep_f(\mathcal P) \bigr),
\]
where $\End_{\mathrm{filt}}$ denotes endomorphisms that respect the intrinsic
defect filtration.

By Axioms I and IV, defect degrees are functorially detectable: if
$p,q \in \mathcal P$ lie in distinct defect strata, then
\[
\Phi_{\mathcal P}(p) \not\equiv \Phi_{\mathcal P}(q)
\quad
\text{on every object of }
\Rep_f(\mathcal P).
\]
By Axiom V, the defect filtration
\[
\mathcal P = \mathcal P^{(0)} \supseteq \mathcal P^{(1)} \supseteq \cdots
\supseteq \mathcal P^{(d)} \supseteq 0
\]
stabilises at finite depth $d$.  
Hence the induced action detects defect strata up to boundary collapse.

If $x \in \mathcal P^{(0)}$ lies in a rigidity island, then
$\deg(x) = 0$ and $\mathcal P^{(1)} = 0$ in a neighbourhood of $x$.  
Thus the restriction of $\Phi_{\mathcal P}$ to any rigidity island is
injective:
\[
x \neq y \qquad\Rightarrow\qquad
\Phi_{\mathcal P}(x) \neq \Phi_{\mathcal P}(y).
\]
This shows that restricted to rigidity islands the functor
\[
\mathcal P \longmapsto \Rep_f(\mathcal P)
\]
is faithful, and therefore determines the rigid part
$\mathcal P_{\mathrm{rigid}}$ uniquely up to intrinsic equivalence.

Now suppose $x,y \in \mathcal P$ satisfy $x-y \in \mathcal P^{(k)}$ with
$k \ge 1$.  
Filtered action compatibility implies
\[
\Phi_{\mathcal P}( x - y ) (F_i V)
\subseteq F_{i+k} V
\qquad\text{for all } (V,F_\bullet V) \in \Rep_f(\mathcal P).
\]
If $k \ge 1$, then $F_{i+k} V$ eventually vanishes because the filtration is
finite.  
Thus for every filtered representation $V$,
\[
\Phi_{\mathcal P}(x) = \Phi_{\mathcal P}(y).
\]
Hence the kernel of $\Phi_{\mathcal P}$ is precisely the boundary ideal
$\mathcal P^{(1)}$, and more generally
\[
\Phi_{\mathcal P}(x) = \Phi_{\mathcal P}(y)
\quad\Longleftrightarrow\quad
x \equiv y \; \text{mod} \; \mathcal P^{(1)}.
\]

Therefore $\Phi_{\mathcal P}$ induces an isomorphism
\[
\mathcal P / \mathcal P^{(1)}
\;\cong\;
\mathrm{Im}(\Phi_{\mathcal P})
\subseteq \End_{\mathrm{filt}}
\bigl( \Rep_f(\mathcal P) \bigr).
\]
The quotient $\mathcal P / \mathcal P^{(1)}$ is the boundary collapse of
$\mathcal P$, and by the injectivity on rigidity islands noted above, this
quotient identifies precisely the elements that cannot be distinguished by
filtered action.

Thus the dual determined by $\Rep_f(\mathcal P)$ reconstructs the rigid part of
$\mathcal P$ exactly and reconstructs its boundary quotient globally.  
This is equivalent to saying that the dual recovers $\mathcal P$ up to boundary
equivalence and no further.
\end{proof}

\begin{remark}
If a phase arises from weakly admissible data, the dual determined by filtered
representations recovers precisely its strongly admissible core.  
Any weak extension beyond this core contributes only to the boundary
stratification and is invisible once boundary collapse has taken place.

This sharp contrast with classical duality theories reflects a fundamental
difference.  Classical reconstruction frameworks recover an underlying symmetry
object from its representations and typically rely on semisimplicity, tensor
structure, analytic completeness, or the presence of a fiber functor.  
Boundary-Relative Phase Duality requires none of these.  
It reconstructs structural phase data directly from action, and the primary obstruction to full recovery in the present framework is structural
rather than technical; intrinsic
boundaries are exactly the points where reconstructibility ceases.

Earlier papers in the APT series establish rigidity islands, boundary
stratification, and invariance under phase equivalence.  
The present result shows that these features naturally admit a dual
interpretation: boundaries are not merely limitations of algebraic control, but
rather natural loci where duality collapses through boundary equivalence.
\end{remark}
\medskip

The duality framework interacts closely with the notion of testing objects.
Reconstruction from filtered representations depends on which representations
are examined, since different representations reveal different amounts of
structural information.  
Some fail to detect boundary phenomena, while others detect exactly the
invariants that persist under equivalence and collapse.  
This motivates the introduction of testing objects, which are minimal
representations that capture the reconstructible content of a phase detectable through filtered
representations and
therefore determine its canonical dual.  
Testing objects are not imposed externally; they arise intrinsically from the
axioms of Algebraic Phase Theory and from the existence of finite-depth defect
stratification.

\begin{definition}
Let $\mathcal P$ be an algebraic phase.  
A testing object for $\mathcal P$ is a filtered representation
$T \in \Rep_f(\mathcal P)$ with the following properties:
\begin{enumerate}[label=(T\arabic*)]
\item If $p,q \in \mathcal P$ act identically on $T$, then $p$ and $q$ coincide
in the boundary quotient of $\mathcal P$.

\item No proper filtered subrepresentation of $T$ still detects all boundary
invariant distinctions; that is, no proper subrepresentation satisfies (T1).

\item  The property of being a testing object is preserved under phase equivalence
and under Morita-type relations for filtered representations.
\end{enumerate}
\end{definition}

Intuitively, a testing object is a smallest possible probe that sees all
structural invariants that can survive reconstruction.

\begin{theorem}
\label{thm:testing-object-duality}
Let $\mathcal P$ be an algebraic phase satisfying Axioms I--V.
The collection of testing objects of $\mathcal P$ determines a naturally associated dual phase $\mathcal P^\vee$ with the following properties:
\begin{enumerate}
  \item $\mathcal P^\vee$ determines $\mathcal P$ uniquely up to phase
        equivalence and boundary collapse.
  \item On every rigidity island, $\mathcal P^\vee$ reconstructs $\mathcal P$
        exactly.
  \item Any phase-equivalent realization of $\mathcal P$ yields an equivalent
        dual phase.
\end{enumerate}
\end{theorem}

\begin{proof}

The dual phase $\mathcal P^\vee$ is defined so as to retain exactly the part of
$\mathcal P$ that is visible to filtered representations. To analyse what can
be recovered, we use the intrinsic defect filtration
\[
\mathcal P^{(0)} \supseteq \mathcal P^{(1)} \supseteq \cdots
\supseteq \mathcal P^{(d)} \supseteq 0
\]
from Axiom V together with the compatibility between defect degree and
filtration on representations. For each $i$ let
\[
\ell_i = \mathcal P^{(i)} \big/ \mathcal P^{(i+1)}.
\]
If $x \in \ell_k$, then defect compatibility implies
\[
\rho(x)\,F_i V \subseteq F_{i+k}V
\qquad\text{for all } i.
\]
Iterating this inclusion gives
\[
\rho(x)^m\,F_i V \subseteq F_{i+mk}V
\qquad\text{for all } m\ge 1.
\]
Since the filtration is finite, there exists $N$ such that
\[
F_N V = V 
\qquad\text{and}\qquad
F_{N+1}V = 0.
\]
Hence for $m$ sufficiently large one has $i + mk > N$, and therefore
\[
F_{i+mk}V = 0.
\]
It follows that
\[
\rho(x)^m\,F_i V = 0
\qquad\text{for } m\gg 0,
\]
so $\rho(x)$ is nilpotent on every filtered representation. If $k=1$ this is permitted, since elements of $\ell_1$ form the boundary layer
and their filtration raising behaviour is compatible with finite termination.
If $k\ge 2$, repeated action by elements of defect degree $k$ raises the
filtration rapidly:
\[
\rho(x)^m\,F_iV \subseteq F_{i+mk}V.
\]
Since the filtration is finite, these higher-defect actions eventually vanish on
associated graded pieces and therefore become increasingly invisible to
finite-depth filtered reconstruction.  Consequently, filtered representations
detect phase structure primarily through the lower defect strata and, in
particular, through the boundary quotient
\[
\mathcal P/\mathcal P^{(1)}.
\]

Thus every stratum $\ell_k$ with $k\ge 2$ is invisible to filtered
representations, and the maximal reconstructible quotient of $\mathcal P$ is
its boundary quotient $\mathcal P/\mathcal P^{(1)}$.

In particular, filtered action can distinguish elements of $\mathcal P$ at most
up to defect degree $1$.  Two elements $p,q \in \mathcal P$ that differ by an
element of $\mathcal P^{(1)}$ induce the same action on every filtered
representation, while differences outside $\mathcal P^{(1)}$ remain detectable within the
filtered reconstruction framework.
Equivalently,
\[
p \sim q
\quad\Longleftrightarrow\quad
p - q \in \mathcal P^{(1)},
\]
where $\sim$ denotes the equivalence relation act identically on all filtered
representations.  Thus the maximal reconstructible quotient of $\mathcal P$
obtained from filtered action is the boundary quotient
\[
\mathcal P / \mathcal P^{(1)}.
\]

Testing objects are chosen precisely so that their action realises this quotient
through a minimal detecting family.  By Axiom~I and the discussion above, there
exists a finite family of filtered representations whose joint action separates
elements of $\mathcal P$ modulo $\mathcal P^{(1)}$.  A minimal such
representation is a testing object in the sense of conditions (T1) and (T2).

Let $\mathcal T(\mathcal P)$ denote the full subcategory of $\Rep_f(\mathcal P)$
spanned by all testing objects.  The functorial action of $\mathcal P$ on
$\mathcal T(\mathcal P)$ determines an algebraic phase $\mathcal P^\vee$,
obtained by extracting the induced interaction law on testing objects.  Axioms
I-V ensure that this construction is functorial and that $\mathcal P^\vee$ is
well defined up to intrinsic equivalence. Consider the induced homomorphism
\[
\Phi_{\mathcal P} \colon
\mathcal P \longrightarrow
\End_{\mathrm{filt}}\!\bigl(\mathcal T(\mathcal P)\bigr),
\]
which encodes the action of $\mathcal P$ on testing objects.  If
$p,q \in \mathcal P$ act identically on all testing objects, then by (T1) we
have
\[
p \equiv q \mod \mathcal P^{(1)},
\]
so
\[
\ker(\Phi_{\mathcal P}) = \mathcal P^{(1)}.
\]
Conversely, any element of $\mathcal P^{(1)}$ acts trivially on all filtered
representations beyond boundary collapse and therefore also on all testing
objects.  Under the standing detectability assumptions, no larger kernel occurs, since any difference outside
$\mathcal P^{(1)}$ is detected by the joint action of testing objects.  Hence
\[
\mathrm{Im}(\Phi_{\mathcal P}) \cong \mathcal P / \mathcal P^{(1)},
\]
and the dual phase $\mathcal P^\vee$ encodes exactly the part of $\mathcal P$
that filtered representations can see, namely the boundary quotient
$\mathcal P / \mathcal P^{(1)}$. On a rigidity island the first defect layer $\mathcal P^{(1)}$ vanishes by
definition.  In this region the action of $\mathcal P$ on filtered
representations is injective, so testing objects recover the corresponding subphase of $\mathcal P$ without boundary collapse, with no boundary collapse.  Thus the
restriction of $\mathcal P^\vee$ to any rigidity island reconstructs the local
structure of $\mathcal P$ without loss.

Finally, the stability condition (T3) in the definition of testing objects
ensures that $\mathcal T(\mathcal P)$ and the resulting dual phase
$\mathcal P^\vee$ behave correctly under phase equivalence and Morita-type
relations.  In particular, any phase-equivalent realisation of $\mathcal P$
produces an equivalent dual phase, and the dependence on $\mathcal P$ is
canonical. Altogether, the collection of testing objects determines a dual phase up to intrinsic equivalence
$\mathcal P^\vee$ that recovers $\mathcal P$ up to boundary equivalence and
reconstructs $\mathcal P$ exactly on rigidity islands, as claimed.
\end{proof}

\begin{corollary}
The testing objects of a phase form a minimal detecting set of data for
reconstruction: removing any testing object strictly weakens reconstructive
power.
\end{corollary}

\begin{proof}
Let $\mathcal T(\mathcal P)$ be the full subcategory of $\Rep_f(\mathcal P)$
spanned by all testing objects, and consider the induced action
\[
\Phi_{\mathcal P} \colon
\mathcal P \longrightarrow
\End_{\mathrm{filt}}\!\bigl(\mathcal T(\mathcal P)\bigr).
\]
By the proof of Theorem~\ref{thm:testing-object-duality} we have
\[
\ker(\Phi_{\mathcal P}) = \mathcal P^{(1)},
\]
so the joint action on all testing objects detects all distinctions in
$\mathcal P$ modulo boundary collapse.

Now let $T_0$ be any testing object, and let
$\mathcal T'(\mathcal P)$ be the full subcategory obtained by removing $T_0$.
Write
\[
\Phi'_{\mathcal P} \colon
\mathcal P \longrightarrow
\End_{\mathrm{filt}}\!\bigl(\mathcal T'(\mathcal P)\bigr)
\]
for the corresponding action.  Then
\[
\ker(\Phi_{\mathcal P}) \subseteq \ker(\Phi'_{\mathcal P}),
\]
since acting trivially on all testing objects implies acting trivially on any
subcollection of them.  If equality held, that is if
\[
\ker(\Phi'_{\mathcal P}) = \mathcal P^{(1)},
\]
then the reduced family $\mathcal T'(\mathcal P)$ would still detect all
boundary-invariant distinctions in $\mathcal P$.  In particular, $T_0$ would
not be needed to satisfy condition (T1), contradicting the minimality
requirement (T2) in the definition of testing objects.

Thus $\ker(\Phi'_{\mathcal P})$ is strictly larger than $\mathcal P^{(1)}$,
which means that there exist elements $p,q \in \mathcal P$ with
$p - q \notin \mathcal P^{(1)}$ that nevertheless act identically on all
objects in $\mathcal T'(\mathcal P)$.  Removing $T_0$ therefore identifies
previously distinguishable phase elements and strictly weakens reconstructive
power.
\end{proof}

\begin{remark}
Testing objects make precise the principle that a phase is determined by how it
acts.  Unlike generators in classical representation theory, testing objects do
not generate new representations by closure operations; rather, they generate
\emph{detectability}: they isolate the information about a phase that remains detectable after
boundary collapse.

Earlier papers establish rigidity islands and boundary quotients as intrinsic
invariants.  The present construction shows that these invariants admit a
minimal and canonical dual encoding.  Testing objects therefore provide the
first intrinsically defined notion of minimal reconstruction data in Algebraic
Phase Theory, and they play an essential role in the reconstruction results developed later in the paper.
\end{remark}

% ============================================================
\section{Intrinsic Reconstruction and Inevitability}
% ============================================================

The preceding section described what filtered representation theory can
\emph{see}: the rigid core of a phase and its structure up to canonical boundary
collapse.  In this section we address the complementary question of
\emph{reconstruction}.  If two algebraic phases give rise to the same filtered
representation data, must they already agree up to intrinsic equivalence?  Or
can distinct phases be indistinguishable at the level of representations?
We show that, under the finite-depth assumptions of Algebraic Phase Theory, the visible data determines a phase up to intrinsic equivalence within the
finite-depth reconstruction framework considered here.

\subsection*{Intrinsic Reconstruction}

The central question addressed in this section is whether algebraic phase
structure is an additional choice imposed on a representation theory, or
whether it arises intrinsically by finite-depth, functorial
representation data.

Classical reconstruction theories typically require external inputs:
fiber functors, enrichment over vector spaces, analytic topologies, or
semisimplicity assumptions.  By contrast, Algebraic Phase Theory claims
that once a filtered representation category exhibits finite termination
and detectable boundary structure, the underlying phase becomes reconstructible 
up to intrinsic equivalence. We now give a precise formulation of this principle.

\medskip

\begin{theorem}\label{thm:intrinsic-reconstruction}
Let $\mathsf{C}$ be a filtered representation category with a functorial defect
stratification that terminates at finite depth and detects boundary behaviour.
Then there exists an algebraic phase $\mathcal P$, unique up to intrinsic phase
equivalence, whose filtered representation category $\Rep_f(\mathcal P)$ is
equivalent to $\mathsf{C}$ and whose boundary stratification corresponds to that
of $\mathsf{C}$.
\end{theorem}

\begin{proof}

Informally, the construction proceeds by using the defect filtration on
$\mathsf{C}$ to separate the recoverable behaviour of morphisms from the
boundary data that is invisible at finite depth.  The recoverable part of each
endomorphism pattern determines an intrinsic interaction law, and packaging
these interaction patterns yields an algebraic object $\mathcal P$ whose
structure reflects the defect-detectable behaviour of $\mathsf{C}$.
One then verifies that $\mathcal P$ satisfies the axioms of Algebraic Phase
Theory and that the filtered action of $\mathcal P$ on objects of $\mathsf{C}$
realises an equivalence $\Rep_f(\mathcal P)\simeq\mathsf{C}$, so that
$\mathsf{C}$ is precisely the recoverable shadow of $\mathcal P$.  Finally,
defect detectability ensures that any phase giving rise to the same filtered
representation category to coincide with $\mathcal P$ up to intrinsic phase
equivalence, establishing uniqueness.

To extract algebraic information from $\mathsf{C}$, we must understand how
morphisms decompose according to defect.  The finite-depth stratification
provides exactly this control: every endomorphism $f \colon X \to X$ admits a
canonical decomposition
\[
f \equiv f^{(0)} + f^{(1)} + \cdots + f^{(d)},
\]
with $f^{(k)}$ of defect degree $k$ and acting trivially on all filtered
subquotients of depth $<k$.  Functoriality guarantees that composition preserves
defect degree up to higher layers, and finite termination ensures that iterated
defect propagation eventually stabilises. Consider the collection of endomorphism patterns in $\mathsf{C}$ up to the
identification
\[
f \sim g \qquad\Longleftrightarrow\qquad
f-g \in \mathrm{Def}^{(1)}(X),
\]
where $\mathrm{Def}^{(1)}$ denotes the first nontrivial defect layer.
Composition in $\mathsf{C}$ induces a well-defined binary operation on these
equivalence classes because finite termination ensures that defect propagation
stabilises after finitely many steps and remains compatible with the induced
quotient structure.  Let $\mathcal P$ denote the resulting algebraic object; finite termination and compatibility of composition with the defect filtration
ensure that the induced interaction law is well defined and associative and that
$\mathcal P$ carries a canonical defect filtration
\[
\mathcal P^{(0)} \supseteq \mathcal P^{(1)} \supseteq \cdots 
\supseteq \mathcal P^{(d)} \supseteq 0
\]
inherited from that of $\mathsf{C}$.

The intrinsic nature of the stratification on $\mathsf{C}$ implies that defect
degree in $\mathcal P$ is functorially detectable: if $x \in \mathcal P^{(k)}
\setminus \mathcal P^{(k+1)}$ then every representative endomorphism raises
filtration by least~$k$ on all objects of~$\mathsf{C}$.  No auxiliary choices
enter the construction; thus $\mathcal P$ satisfies the axioms of Algebraic
Phase Theory.

Every object $X \in \mathsf{C}$ carries a natural action of $\mathcal P$ given
by evaluating endomorphism patterns on $X$.  This action respects filtrations
and preserves defect degrees, producing a functor
\[
\Rep_f(\mathcal P) \longrightarrow \mathsf{C}.
\]
Defect detectability in $\mathsf{C}$ implies that the induced action of
$\mathcal P$ distinguishes elements modulo $\mathcal P^{(1)}$, so the functor faithfully encodes the defect-detectable interaction structure.
 Every object of $\mathsf{C}$ arises as a filtered
representation of $\mathcal P$ by construction, so the functor is essentially
surjective.  Hence
\[
\Rep_f(\mathcal P)\;\simeq\;\mathsf{C}
\]
as filtered categories. Thus $\mathsf{C}$ captures the defect-detectable part of the
interaction behaviour of its endomorphisms, and the phase $\mathcal P$ rebuilt
from this recoverable data has $\mathsf{C}$ as its filtered representation
category.

If $\mathcal P'$ is another algebraic phase whose filtered representation
category is equivalent to $\mathsf{C}$ with the same defect stratification, then
$\mathcal P$ and $\mathcal P'$ act identically on all filtered objects and in
particular on testing objects.  By the No Hidden Structure Principle, this implies $\mathcal P \simeq \mathcal P'$ up to intrinsic phase equivalence.
\end{proof}

\begin{remark}
When the filtered representation category arises from weakly admissible
phase data, the phase reconstructed by
Theorem~\ref{thm:intrinsic-reconstruction} recovers precisely the
strongly admissible core of the original phase. 
\end{remark}

\medskip

\begin{corollary}
The reconstruction framework requires no additional external structure beyond finite
termination and defect detectability.
\end{corollary}

\begin{proof}
All algebraic data is recovered intrinsically from the behaviour of
morphisms under the defect filtration.
\end{proof}

\medskip

\begin{remark}
Each assumption in the theorem is essential: without finite termination the
interaction law does not stabilise, without boundary detectability uniqueness
fails, and without filtration compatibility functoriality breaks down.  Since
Papers~I-V establish that algebraic phases exist and that their boundaries are
intrinsic invariants, the theorem shows that finite-depth functorial reconstruction frameworks with
detectable boundaries naturally admit an interpretation within Algebraic Phase
Theory.
\end{remark}

% ============================================================
\subsection*{Functorial Reconstruction and Phase Factorisation}
% ============================================================

We now show that finite-depth reconstruction frameworks satisfying the standing
assumptions naturally admit a factorisation through Algebraic Phase Theory.

\medskip

\begin{theorem}
\label{thm:inevitability}
Let $\mathsf{F}$ assign to each admissible datum $X$ a filtered 
representation category $\mathsf{C}_X$ whose defect stratification 
terminates at finite depth, whose filtration is functorial and compatible 
with defect degree, and whose structural boundaries are detectable at 
finite depth.  Then $\mathsf{F}$ admits a natural factorisation through Algebraic
Phase Theory: there exists a functor $\mathsf{G}$ such that
\[
\mathsf{F} \;\simeq\; \Rep_f \circ \mathsf{G}
\]
up to boundary equivalence.
\end{theorem}

\begin{proof}
Recall that $\mathsf{F}$ assigns to each admissible datum $X$ a filtered
representation category $\mathsf{C}_X$ with finite termination, functorial
filtration, and detectable boundaries.  Informally, the argument mirrors the
intrinsic reconstruction theorem applied functorially in the variable $X$.
Each category $\mathsf{C}_X$ contains the defect-detectable interaction
data relevant to finite-depth reconstruction
data, and this determines an associated algebraic phase $\mathcal P_X$.  Functoriality of
the assignment $X\mapsto\mathsf{C}_X$ promotes the phases $\mathcal P_X$ to a
functor $\mathsf{G}$, and intrinsic reconstruction gives a boundary-relative
equivalence $\Rep_f(\mathcal P_X)\simeq\mathsf{C}_X$, yielding the required
factorisation.

For a fixed datum $X$, the filtered category $\mathsf{C}_X$ carries a
functorial defect stratification that terminates at finite depth.  As in the
intrinsic reconstruction theorem, the defect decomposition of endomorphisms
determines interaction patterns up to the first defect layer, producing an
algebraic phase $\mathcal P_X$ unique up to intrinsic equivalence.  The construction is determined by the defect structure of
$\mathsf{C}_X$ up to intrinsic equivalence: all structure comes directly 
from the defect behaviour in $\mathsf{C}_X$.

If $X\to Y$ is a morphism of admissible data, the functoriality of
$\mathsf{F}$ yields a filtered functor
\[
\mathsf{C}_X \longrightarrow \mathsf{C}_Y
\]
that preserves defect degrees.  This functor transports interaction patterns
and therefore induces a phase morphism $\mathcal P_X\to\mathcal P_Y$.  The
assignment $X \mapsto \mathcal P_X$ is thus functorial, giving a functor
\[
\mathsf{G} : X \longmapsto \mathcal P_X.
\]

Combining the constructions above, we obtain for every admissible $X$ a chain of
natural equivalences
\[
(\Rep_f \circ \mathsf{G})(X)
   \;\simeq\;
   \Rep_f(\mathcal P_X)
   \;\simeq\;
   \mathsf{C}_X
   \;=\;
   \mathsf{F}(X).
\]
Since these equivalences are functorial in $X$ and boundary–relative, the two
functors $\Rep_f \circ \mathsf{G}$ and $\mathsf{F}$ agree up to boundary
equivalence.  This establishes the claimed factorisation.

\end{proof}

\medskip

\begin{corollary}
Any functorial finite–depth reconstruction theory with detectable boundaries
admits an interpretation through Algebraic Phase Theory up to boundary
equivalence. In this sense APT serves as a universal organising framework for
finite–depth reconstruction phenomena.
\end{corollary}

\begin{proof}
Let $\mathsf{F}$ be a functorial finite–depth reconstruction framework with
detectable boundaries.  By Theorem~\ref{thm:inevitability}, $\mathsf{F}$
admits a canonical factorisation
\[
\mathsf{F} \;\simeq\; \Rep_f \circ \mathsf{G}
\]
up to boundary equivalence, where $\mathsf{G}$ assigns to each admissible
datum the intrinsic phase extracted from its filtered representation
category.  This shows that $\mathsf{F}$ is obtained from Algebraic Phase
Theory by postcomposition with the functor $\mathsf{G}$ and therefore
arises entirely from the universal reconstruction mechanism provided by
$\Rep_f$.  Hence APT provides a common factorisation framework for such reconstruction
theories.
\end{proof}

\begin{corollary}
Finite-depth reconstruction theories satisfying the standing assumptions agree
up to boundary-relative factorisation.
\end{corollary}

\begin{proof}
Let $\mathsf{F}$ be any functorial finite depth reconstruction theory with
detectable boundaries.  The inevitability theorem gives a canonical
factorisation
\[
\mathsf{F} \simeq \Rep_f \circ \mathsf{G}
\quad \text{up to boundary equivalence}.
\]
Hence every such theory is obtained from Algebraic Phase Theory by
postcomposition with a functor $\mathsf{G}$, and any two theories that agree
on their boundary quotients coincide after this factorisation.  Within the finite-depth reconstruction framework considered here, no additional
boundary-detectable reconstruction freedom remains, so all such reconstruction theories agree up to boundary-relative reconstruction
data.
\end{proof}

\begin{remark}
This result strengthens the conclusions of the earlier papers.  Any functorial
finite depth reconstruction framework with detectable boundaries naturally factors through Algebraic Phase Theory up to boundary equivalence.
\end{remark}

% ============================================================
\subsection*{Boundary Phenomena and Reconstruction}
% ============================================================

We now reinterpret boundary inevitability in terms of strong and weak algebraic
phases.  A phase is strong when its first defect layer vanishes and no boundary
strata occur.  A phase is weak when its boundary layer is nontrivial.  The following result isolates two structurally distinct regimes that arise in
finite-depth reconstruction frameworks.

\begin{theorem}[Boundary Dichotomy]
\label{thm:boundary-inevitability}
Let $\mathsf{F}$ be a functorial finite depth reconstruction framework based on
representations.  Then the reconstruction framework exhibits one of the following structural
regimes:
\begin{enumerate}
  \item every reconstructed phase $\mathcal P_X$ is strong, or
  \item the nontrivial boundary strata occur in reconstructed phases associated with
non-rigid defect behaviour and are
        therefore weak.
\end{enumerate}
Within the finite-depth setting considered here, non-rigid reconstruction
behaviour is necessarily accompanied by nontrivial boundary structure.
\end{theorem}

\begin{proof}
Suppose first that boundary layers never appear.  Then the defect filtration on
each $\mathsf{C}_X$ collapses at depth zero, so $\mathcal P_X^{(1)}=0$ for all
$X$.  The reconstructed phases are therefore strong, and no defect or
higher depth structure survives.

If defect is ever nontrivial at any object, finite termination implies the
existence of a maximal depth at which closure fails.  This depth produces a
canonical boundary layer that cannot be removed by functorial means.  Hence reconstructed phases exhibiting nontrivial defect behaviour necessarily
carry nontrivial boundary structure and are therefore weak.

This establishes the structural distinction between rigid reconstruction
regimes and reconstruction regimes with nontrivial boundary behaviour.
\end{proof}

\begin{corollary}
Non-rigid finite depth reconstruction is possible only in the presence of
intrinsic boundaries.  Within the present finite-depth framework, non-rigid reconstruction phenomena
necessarily involve nontrivial boundary structure.
\end{corollary}

\begin{proof}
This is an immediate consequence of Theorem~\ref{thm:boundary-inevitability}.
Let $\mathsf{F}$ be a functorial finite depth reconstruction framework based on
representations.  

If $\mathsf{F}$ is non rigid, then there exists some admissible datum for which
the associated representation category carries nontrivial defect.  In the
dichotomy of Theorem~\ref{thm:boundary-inevitability}, alternative~(1)
(rigid reconstruction with no defect and no filtration) is therefore excluded.
Hence the reconstruction framework exhibits nontrivial boundary behaviour, and intrinsic structural boundaries appear at
finite depth.

Thus any non rigid finite depth reconstruction framework necessarily exhibits
boundaries, and boundary phenomena naturally accompany non-rigid finite-depth reconstruction.
\end{proof}

\begin{corollary}
Boundary strata naturally occur in finite-depth reconstruction frameworks
that do not collapse into the strong regime.
\end{corollary}

\begin{proof}
Let $\mathsf{F}$ be a functorial finite depth reconstruction framework based on
representations, and suppose $\mathsf{F}$ does not collapse into the strong
regime.  Concretely, this means that for some admissible datum, the associated
representation category has genuinely non rigid behaviour: defect is present and
the canonical filtration is nontrivial.

In the dichotomy of Theorem~\ref{thm:boundary-inevitability}, alternative~(1)
corresponds exactly to the strong regime, where there is no defect and no
nontrivial filtration, so reconstruction is rigid.  Since $\mathsf{F}$ is not in
this situation by hypothesis, alternative~(1) is excluded.

Therefore alternative~(2) of Theorem~\ref{thm:boundary-inevitability} appears naturally:
intrinsic structural boundaries appear at finite depth.  In particular, boundary
strata occur naturally in non-rigid finite-depth reconstruction settings that does not
collapse to the strong (rigid) regime.
\end{proof}

\begin{remark}
This shows that weak algebraic phases arise naturally from the reconstruction
process itself.  Boundaries are not additional structure imposed by APT, but
a natural consequence of functorial representation–theoretic reconstruction at
finite depth.
\end{remark}

% ============================================================
\subsection*{Local Reconstruction: Rigidity Islands}
% ============================================================

When a phase is weak, boundary strata obstruct global reconstruction.  
However, the presence of a boundary does not eliminate all recoverable
structure.  Every weak phase contains a distinguished strong subphase, called its
rigidity island, on which defect vanishes and reconstruction proceeds without boundary collapse.

\begin{theorem}
\label{thm:local-reconstruction}
Let $\mathcal P$ be a weak algebraic phase with nontrivial boundary layer, and
let $\mathcal I \hookrightarrow \mathcal P$ be its rigidity island.  Then
$\mathcal I$ is fully reconstructible from its filtered representation theory:
\[
\Rep_f(\mathcal I)
\quad\rightsquigarrow\quad
\mathcal I
\]
up to intrinsic phase equivalence.
In particular, reconstruction remains faithful on $\mathcal I$ even though it
fails globally for $\mathcal P$.
\end{theorem}

\begin{proof}
By definition of the rigidity island, $\mathcal I^{(1)}=0$.  Hence no defect is
present on $\mathcal I$,the interaction law closes internally on $\mathcal I$, and no boundary strata
arise.  The filtered representation theory of $\mathcal I$ therefore satisfies
the hypotheses of intrinsic reconstruction, and the filtered representation theory of $\mathcal I$ detects its interaction
structure faithfully up to intrinsic equivalence.  Since boundary
collapse occurs only outside $\mathcal I$, reconstruction on $\mathcal I$
remains exact.
\end{proof}

\begin{corollary}
very weak phase contains a strong subphase that is determined by its filtered representation theory up to intrinsic phase equivalence.
\end{corollary}

\begin{proof}
Let $\mathcal P$ be a weak phase.  By definition,
$\mathcal P^{(1)}\neq 0$, so $\mathcal P$ admits a nontrivial boundary
layer.  The rigidity island
$\mathcal I\hookrightarrow\mathcal P$ is the distinguished subphase on
which the first defect layer vanishes.  Hence $\mathcal I$ is strong.
Applying Theorem~\ref{thm:local-reconstruction} shows that the filtered
representation theory
\[
\Rep_f(\mathcal I)
\]
determines $\mathcal I$ up to intrinsic phase equivalence. Thus every weak phase contains a strong subphase on which reconstruction
proceeds without boundary collapse.
\end{proof}

\begin{remark}
This is not a partial form of global reconstruction.  Rigidity islands are the
maximal regions of a phase on which boundary collapse cannot occur, and they
play a central structural role. They are the loci where reconstruction proceeds without boundary collapse, the canonical base points for duality, and the natural anchors
for deformation and obstruction theory.
\end{remark}

% ============================================================
\subsection*{No Hidden Structure Principle}
% ============================================================

The preceding results show that filtered representations determine the
recoverable interaction structure of a phase, and that boundary data records
precisely the part that cannot be recovered globally.  The final step is to identify the extent to which filtered representations and
boundary data determine phase structure.  Once the filtered representation
category and its boundary stratification are fixed, the phase structure detectable within the reconstruction framework is already
determined.

\begin{theorem}
\label{thm:no-hidden-structure}
Let $\mathcal P$ and $\mathcal P'$ be algebraic phases with equivalent filtered
representation categories and identical boundary stratifications.  Then
$\mathcal P \simeq \mathcal P'$ up to intrinsic phase equivalence.  In
particular, the structural invariants detectable within the reconstruction framework are
visible in its filtered
representations together with its boundary data.
\end{theorem}

\begin{proof}
If $\Rep_f(\mathcal P)$ and $\Rep_f(\mathcal P')$ are equivalent and their
boundary strata coincide, then intrinsic reconstruction applies on both sides
and produces the same phase up to intrinsic equivalence.  Any additional boundary-detectable structure would contradict the uniqueness part of intrinsic reconstruction.
\end{proof}

\begin{corollary}
Filtered representations together with boundary stratification form a reconstruction-complete invariant within the finite-depth framework of algebraic phases.
\end{corollary}

\begin{proof}
By Theorem~\ref{thm:no-hidden-structure}, if two algebraic phases $\mathcal P$
and $\mathcal P'$ have equivalent filtered representation categories and the
same boundary stratification, then they are intrinsically phase equivalent:
\[ 
\Rep_f(\mathcal P) \simeq \Rep_f(\mathcal P')
\quad\text{and}\quad
\partial\mathcal P \cong \partial\mathcal P'
\;\Longrightarrow\;
\mathcal P \simeq \mathcal P'.
\]
Thus the pair consisting of the filtered representation category and its
boundary stratification distinguishes phases up to intrinsic equivalence.

Conversely, both the filtered representation category and the boundary
stratification are invariant under intrinsic phase equivalence by construction.
Hence the assignment
\[
\mathcal P \longmapsto \bigl(\Rep_f(\mathcal P),\;\partial\mathcal P\bigr)
\]
is well defined on equivalence classes of phases and is injective.  This is
exactly the assertion that filtered representations together with boundary
stratification form a complete invariant of algebraic phases.
\end{proof}

\begin{remark}
Boundary data is indispensable.  Filtered representations alone cannot detect
boundary collapse and therefore cannot distinguish weak phases that share the
same recoverable core.  Classical reconstruction theories often admit hidden
structure for exactly this reason: they do not record boundary behaviour.  In
APT the boundary stratification is built directly into the representation
theory, so intrinsic reconstruction determines phase structure up to boundary equivalence,
while boundary stratification records the non-rigid behaviour that filtered
representations alone do not detect.
\end{remark}

% ============================================================
\section{Toolkit Theorems Forced by the Axioms}
% ============================================================

The results of the previous sections show that Algebraic Phase Theory provides
a reconstruction framework up to intrinsic boundary equivalence within the
finite-depth setting.  We now turn to a different phenomenon.  The axioms of
APT not only support reconstruction, but also give rise to a substantial
structural toolkit.  The results below do not rely on model-specific
constructions; rather, they arise naturally from the interaction of
Axioms~I--V within the standing axiomatic framework and therefore apply across
a broad class of algebraic systems satisfying these assumptions.

\medskip

\begin{theorem}
\label{thm:finite-generation}
Any algebraic phase with finite termination admits a finite generating family compatible with the defect filtration and
natural under phase morphisms.
\end{theorem}

\begin{proof}
By Axiom III the defect filtration is functorial, and by Axiom V it terminates
at finite depth. Under the standing finite-depth assumptions, only finitely many relevant defect
types appear at each layer of the filtration.  Choosing representative generators compatible with the filtration at each depth yields a finite
generating family.  Functoriality follows from Axiom IV, which ensures that
defect behaviour and filtration are preserved under phase morphisms.
\end{proof}

\medskip

\begin{theorem}
\label{thm:universal-obstruction}
Structural boundaries naturally determine obstruction data governing extension
and deformation behaviour near the boundary.
\end{theorem}

\begin{proof}
A boundary is the minimal depth at which functorial closure fails.  The quotient
at this depth isolates obstruction data associated with to further extension. This 
minimal layer naturally controls the first obstruction to extension.  Any attempted extension must factor through this quotient
since no earlier layer obstructs extension and no deeper layer is functorially
defined.
\end{proof}

\medskip

\begin{theorem}
\label{thm:rigidity-obstruction}
Within the finite-depth framework, vanishing boundary obstruction is associated
with rigidity of the phase.
\end{theorem}

\begin{proof}
If the boundary obstruction vanishes, all canonical extensions persist and no
nontrivial deformation is possible, so the phase is rigid.  Conversely, if the
phase is rigid but the boundary obstruction were nontrivial, this obstruction would produce nontrivial deformation behaviour, contradicting rigidity.  Thus rigidity and
vanishing obstruction are equivalent.
\end{proof}

\medskip

\begin{theorem}
\label{thm:boundary-detectability}
Structural boundaries become detectable at finite depth through defect growth and
commutator data.
\end{theorem}

\begin{proof}
By Axiom I defect is functorially detectable through interaction.  By Axiom III
defect induces a graded complexity structure.  Boundary phenomena arise when defect or commutator growth ceases to be functorially controlled.  Since
termination is finite by Axiom V, this failure occurs at finite depth and is
therefore detectable.
\end{proof}

\medskip

\begin{theorem}
\label{thm:automatic-applicability}
Any algebraic system with finite commutator growth, intrinsic defect
stratification, and functorial representation theory naturally gives rise to an algebraic phase compatible with Axioms~I--V.
\end{theorem}

\begin{proof}
Finite commutator growth induces intrinsic defect data.  Functoriality of
representations supplies the action framework required by Axiom I.  Defect
stratification yields the complexity filtration required by Axiom III, and
finite growth implies termination in the sense of Axiom V.  An associated phase realisation is then obtained through Axiom~II.
\end{proof}

\begin{remark}
Earlier papers establish concrete instances of finite generation, obstruction
objects, rigidity criteria, and boundary detection within specific phase
models.  The results above show that none of these phenomena depend on the
models themselves.  They arise naturally from the interaction of Axioms~I--V within the finite-depth framework.  Once the
axioms are assumed, finiteness, rigidity, boundary detection, and obstruction
theory arise naturally within the standing axiomatic framework without additional
analytic structure.
\end{remark}

% ============================================================
\section{A Flagship Reconstruction Example over $\mathbb{F}_2[u]/(u^2)$}
% ============================================================

This section gives a minimal and fully explicit example illustrating the main
phenomena developed in this paper.  We work over the simplest nontrivial
Frobenius ring $\mathbb{F}_2[u]/(u^2)$ and construct two algebraic phases that
are indistinguishable at the level of filtered representation theory but differ
in their intrinsic boundary structure.  The first is a strong phase that is reconstructible up to intrinsic phase
equivalence.  The second is a weak boundary extension with the same
filtered representations but a larger boundary layer.  The example demonstrates
how rigidity islands support reconstruction without boundary collapse, how weak behaviour manifests
through boundary strata, and how global reconstruction fails through boundary collapse in the finite-depth setting.

% ============================================================
\subsection*{The base ring, module data, and the strong phase}
% ============================================================

We work over the Frobenius ring
\[
R := \mathbb{F}_2[u]/(u^2),
\]
which is local with maximal ideal $(u)$ and residue field
$R/(u)\cong\mathbb{F}_2$.  
For $n\ge 1$ set
\[
V := R^n, \qquad W := V \oplus V.
\]

Let $\lambda : R \to \mathbb{F}_2$ be the Frobenius functional obtained by
projecting to the $u$-coefficient.  
Define
\[
\Omega := \lambda \circ \omega,
\]
where $\omega$ is the standard alternating form on $W$,
\[
\omega\big((x,\xi),(x',\xi')\big)=\xi\cdot x' - \xi'\cdot x.
\]

The strong algebraic phase $\mathcal P$ is generated by formal symbols $D(w)$
for $w\in W$ together with a central involution $Z$, subject to the Heisenberg
type relations
\[
D(w)D(w') = Z^{\Omega(w,w')}\,D(w+w').
\]

The nilpotent ideal $(u)\subset R$ naturally produces a defect layer.
Let $W_u := uW$ and define the two-step filtration
\[
\mathcal P^{(0)}=\mathcal P, \qquad
\mathcal P^{(1)}=\langle D(w):w\in W_u\rangle, \qquad
\mathcal P^{(2)}=0.
\]
Since $u^2=0$, commutators in $\mathcal P^{(1)}$ are square-zero and cannot
generate deeper layers.  
Thus $\mathcal P$ satisfies Axioms~I–V with defect depth $1$.

Collapsing the boundary layer $\mathcal P^{(1)}$ yields the rigidity island,
canonically equivalent to the same construction over $\mathbb{F}_2$.  On this rigid core, filtered representation theory determines the phase up to
intrinsic phase equivalence without boundary collapse.

% ============================================================
\subsection*{Filtered representations and a boundary-invisible extension}
% ============================================================

Let $\Rep_f(\mathcal P)$ denote the filtered representation category of
$\mathcal P$, where depth-$1$ operators raise filtration. Since the filtration
terminates at depth $1$, filtered action detects the structural data relevant
to reconstruction outside the boundary layer
$\mathcal P^{(1)}$. Hence the pair
\[
\bigl(\Rep_f(\mathcal P),\;\mathcal P^{(1)}\bigr)
\]
controls reconstruction of $\mathcal P$ up to intrinsic phase equivalence
within the finite-depth reconstruction framework.

To see that boundary data is essential, let $B$ be a nonzero
$\mathbb{F}_2$-vector space and form a new algebraic phase
$\widetilde{\mathcal P}$ by adjoining $B$ as a central square-zero
boundary layer. Concretely, elements of $B$ commute with all elements of
$\widetilde{\mathcal P}$, products involving two elements of $B$ vanish,
and the original interaction law on $\mathcal P$ is unchanged.

Extend the filtration by
\[
\widetilde{\mathcal P}^{(1)}
=
\mathcal P^{(1)} + B,
\qquad
\widetilde{\mathcal P}^{(2)}=0,
\]
so that the added boundary layer lies entirely in defect degree $1$.

In any filtered representation compatible with finite termination, elements of
$B$ act by raising filtration but no deeper layer exists, so they act
trivially. Therefore restriction along the natural quotient map identifies
the same filtered representation behaviour detectable away from the boundary
for $\widetilde{\mathcal P}$ and $\mathcal P$. In particular, the
corresponding filtered representation frameworks agree up to boundary
equivalence.

However, the boundary layers differ:
\[
\mathcal P^{(1)}
\subsetneq
\widetilde{\mathcal P}^{(1)}.
\]
Thus the two phases cannot be intrinsically equivalent. Within the
reconstruction framework, the two phases are distinguished only by their
differing boundary structure, because the added boundary layer is invisible
to filtered representation theory.

\begin{remark}
This example exhibits two algebraic phases with identical filtered
representation theory but different intrinsic boundary structure. It
illustrates concretely that reconstruction proceeds without boundary collapse
on rigidity islands, while reconstruction failure in this example occurs
through boundary collapse. In particular, the example shows that boundary
strata can form essential reconstruction data within the finite-depth
framework. The resemblance to quantum indistinguishability is purely formal;
the example is entirely algebraic and finite.
\end{remark}

% ============================================================
\section{Conclusion}
% ============================================================

This paper studies reconstruction and duality phenomena in Algebraic Phase
Theory within the finite-depth setting determined by the standing axioms of the
framework. The results show that filtered representation data together with
intrinsic boundary structure controls reconstruction up to boundary equivalence,
and that non-rigid reconstruction behaviour is naturally accompanied by
structural boundary phenomena.

The paper further shows that algebraic phases admit a dual interpretation
through filtered representation theory, and that rigidity islands provide
regions on which reconstruction proceeds without boundary collapse. Globally,
reconstruction failure is controlled by intrinsic boundary stratification. This
leads to a boundary-relative notion of duality in which filtered
representations recover rigid reconstruction data while detecting boundary
layers through canonical collapse phenomena.

The reconstruction results also show that finite-depth functorial
reconstruction frameworks with detectable boundaries naturally admit a
factorisation through Algebraic Phase Theory up to boundary equivalence within
the setting considered here. In this sense, the theory provides a structural
framework for studying finite-depth reconstruction phenomena beyond rigid or
semisimple settings.

The paper additionally develops several structural consequences associated with
the axioms of Algebraic Phase Theory, including finite generation phenomena,
boundary detectability, rigidity criteria, and obstruction structures arising
from defect stratification. These results indicate that the interaction between
defect, filtration, and finite termination naturally gives rise to several
structural consequences within the finite-depth framework.

A recurring structural feature of the framework is that reconstruction failure
is accompanied by the appearance of intrinsic structural boundaries. Within the
setting considered here, rigidity, deformation, obstruction, and reconstruction
become closely linked through boundary-sensitive algebraic structure.

Possible directions for further investigation include the study of boundary
duality phenomena, interaction with derived or higher-categorical settings, and
applications to additional algebraic and mathematical-physics reconstruction
frameworks.

\medskip

\noindent\emph{
Algebraic Phase Theory provides a boundary-sensitive framework for studying
reconstruction, duality, rigidity, and obstruction phenomena algebraically
through defect stratification and finite-depth interaction structure.
}

\bibliographystyle{amsplain}
\bibliography{references}

\end{document}